\documentclass[12pt, a4paper]{amsart}

\usepackage{accents,graphicx,hyperref}

\newtheorem{theorem}{Theorem}

\theoremstyle{definition}

\theoremstyle{plain}
\newtheorem{thm}{Theorem}[section]
\newtheorem{lem}[thm]{Lemma}
\newtheorem{prop}[thm]{Proposition}

\title[Asymptotic Faithfulness of Quantum $\mathrm{Sp}(4)$ Representations]{Asymptotic Faithfulness of Quantum $\mathrm{Sp}(4)$ Mapping Class Group Representations}

\author{Wade Bloomquist}
\email{bloomquist@math.ucsb.edu}
\address{Dept. of Mathematics\\
    University of California\\
    Santa Barbara, CA 93106-6105\\
    U.S.A.}

\begin{document}

\begin{abstract}

We prove asymptotic faithfulness for the quantum $\mathrm{Sp}(4)$ mapping class group representation.  This provides the first example of asymptotic faithfulness lying outside of the $A_n$ family. The methods used are generalized from the proof of asymptotic faithfulness for skein $SU(2)_k$ mapping class group representations.  In short, for any noncentral mapping class a comparison vector is found which allows for the mapping class to be detected.

\end{abstract}

\maketitle

\section{Introduction}

Bridging the fields of classical and quantum topology remains an active area of research.  When looking at mathematical formulations of quantum Witten-Chern-Simons theories, WCS$(G,k)$ \cite{Witt}, it is expected that aspects of classical topological will emerge when the level of the theory, $k$, tends to infinity.  The mathematical realizations of quantum Witten-Chern-Simons theories are commonly considered to be the Reshetikhin-Turaev TQFTs arising from the modular tensor cateogires $G_k$, which are constructed from quantum groups at roots of unity \cite{RT1,RT2,Tur,RSW}.  The construction outlined by Turaev allows for the construction of a TQFT given any modular tensor category.  Of particular interest for us are the Jones-Kauffman TQFTs \cite{Tur, BHMV, Wang}.  Here the modular category in question arises from Temperley-Lieb-Jones algebroids.  The resulting modular tensor category should be seen as parallel to $SU(2)_k$.  These Temperley-Lieb-Jones categories arose from the work of Jones \cite{Jon}, and were reformulated by Kauffman using a diagrammatic language \cite{KL}.   

In the spirit of recovering features of classical topology from these constructions we recall a question originally posed by Turaev \cite{Tur}: 
\newline
\lq\lq Is the (projective) action of the mapping class group of a closed oriented surface $\Sigma$ in the module of states of $\Sigma$ irreducible?  Consider the kernels of these actions corresponding to all modular categories.  Is the intersection of these kernels non-trivial? (This would be hard to believe.)"
\newline
In a more modern language this question is phrased in terms of asymptotic faithfulness.  Namely, does the quantum mapping class group representation arising from $G_k$ (where we all $G_k$ to refer to skein constructions as well) eventually detect non-central elements of the mapping class group as $k$ tends to infinity?
Asymptotic faithfulness for skein $SU(2)_k$, meaning for Jones-Kauffman TQFTs, was proven in \cite{FWW},  which is independent from the  parallel asymptotic faithfulness of the representations from $SU(n)_k$ TQFTs \cite{And}.   The skein theoretic methods were extended to skein $SU(3)_k$ in \cite{BW}.

A salient feature of $SU(2)_k$ TQFTs is the multiplicity-freeness of all fusion spaces: the dimension of the state space for a labeled pair of pants is either $0$ or $1$.  We also have that there is a single self-dual fundamental representation.   It follows that the skein theory requires only unoriented simple closed curves, of a single strand type, and so unoriented trivalent graphs with uncolored trivalent vertices.   Then a complexity of vector in the state space is given the topological interpretation of the maximal number of strands passing through the dual disks to the edges being labeled.  To generalize to $Sp(4)$, we have to use two strand types and through some careful maneuvering are able to keep a very similar notion of complexity.  Once we appropriately set the argument in place, our proof is completely analogous to that of skein $SU(2)_k$.

\section{The $C_2$ Spider}
\subsection{Abstract Spiders}We provide a description of spiders which uses more modern language than in the original definition given by Kuperberg \cite{Kup}.  Given a pivotal tensor category and a collection of objects a spider is the full subcategory whose objects are tensor products of the chosen collection of objects and their duals.  A spider serves the role of a planar algebra with labeled strands.  In particular, if the label set is a single symmetrically self-dual object, meaning it is self dual and it's associated frobenius-schur indicator is $1$, then the associated spider is an unoriented unshaded planar algebra.  Kuperberg's original definition can be seen to capture the essential notions found in a pivotal tensor category. Our interest will be solely in the $B_2/C_2$ spider.  Here we mean the spider generated by the fundamental representations  of $Rep^{uni}(U_q(\mathfrak{sp}(4)))$, where the superscript denotes that the unimodal pivotal structure is taken.
\subsection{The Combinatorial $C_2$ Spider} 
Kuperberg has provided a combinatorial description of this spider.  In particular there are two strand types, one associated to each fundamental representation of $U_q(\mathfrak{sp}(4))$.  As in Kuperberg's original work we will use a single strand and a double strand to diagrammatically represent these two strand types.  Then we have that the $C_2$ spider is generated by a single trivalent vertex type, as seen in Figure $1$.  
\begin{figure}
\includegraphics[scale=.25]{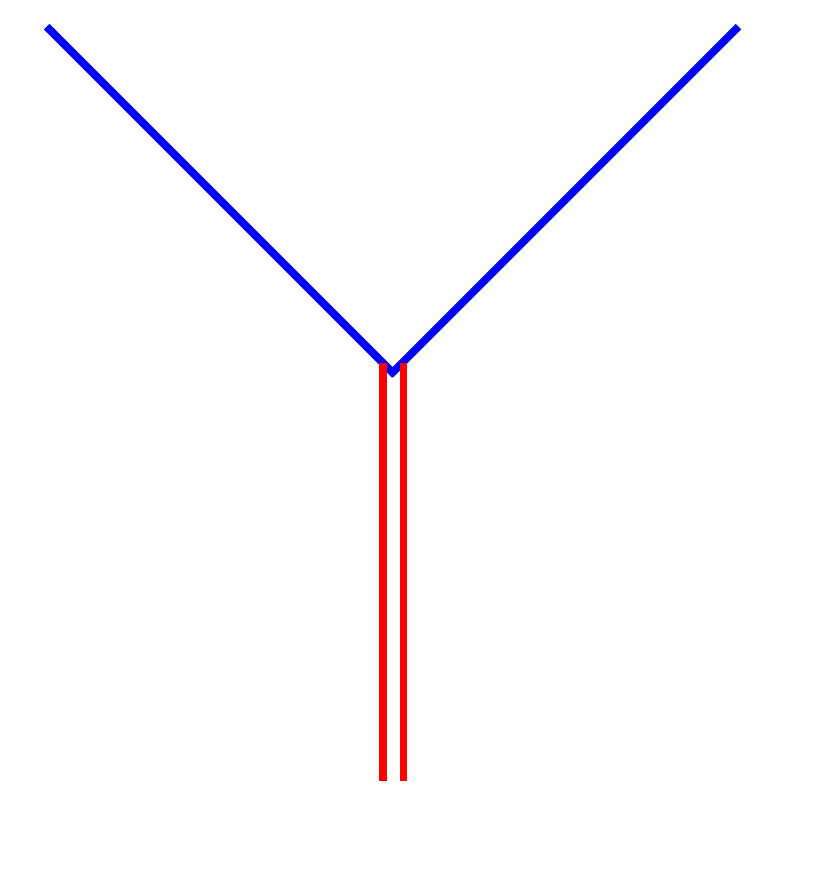}
\caption{The generator of the $C_2$ spider}
\end{figure}
The relations seen in Figure $2$ then complete the description of the $C_2$ spider.
\begin{figure}
\includegraphics[scale=.45]{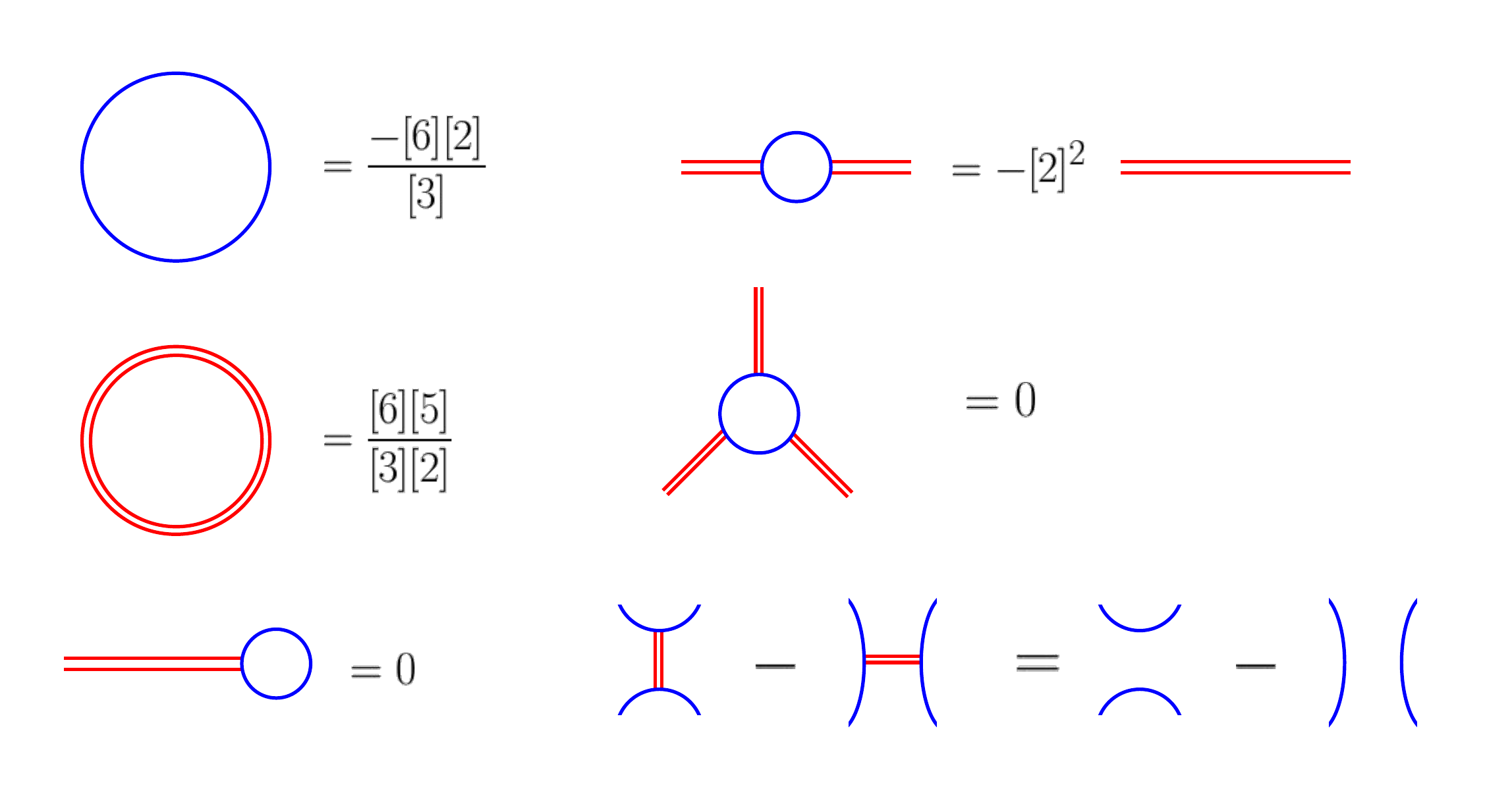}
\caption{The relations in the $C_2$ spider}
\end{figure}
The equivalence of this combinatorial spider with the spider described above implies that the morphisms in the representation category can be described as linear combinations of diagrams generated as above subject to the given relations.  This is analogous to non-crossing planar matchings, or Temperley-Lieb diagrams, for $Rep^{uni}(U_q(sl(2,\mathbb{C})))$.

\subsection{Clasped $C_2$ Spider}  
Clasps provide a method for passing from the spider as described above to a new full subcategory, namely the one generated by all irreducible representations.  The power in taking this additional step is that we can bring the combinatorial description given to the original spider over to this new full subcategory (which is actually of much greater interest).  While the existence of these clasps was proven in Kuperberg's original work, explicit constructions of the clasps, making use of the combinatorial structure, were partially given by Kim \cite{Kim}.  Partial results in this case mean that constructions are only found for clasps of the type $(p,0)$ and $(0,q)$.  In order to best state these results the change of basis seen in Figure $3$ is introduced.
\begin{figure}
\includegraphics[scale=.55]{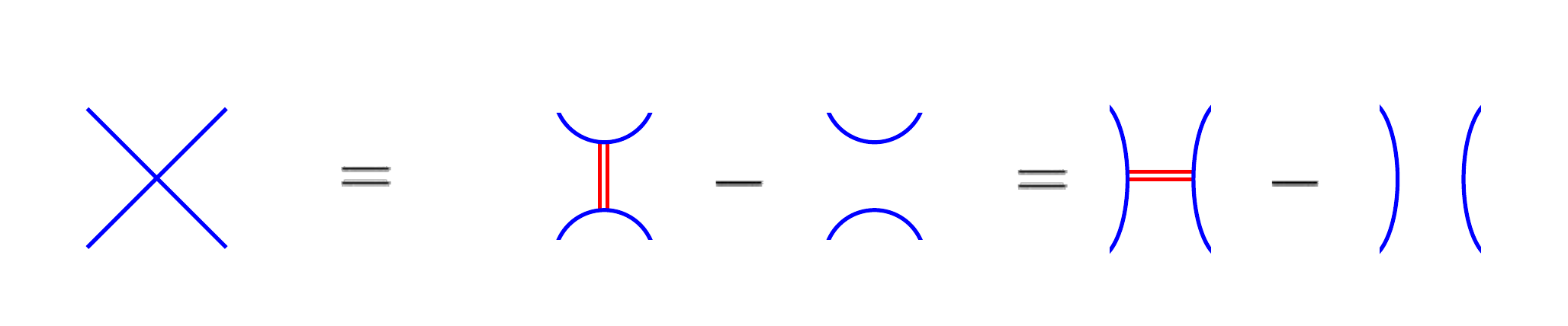}
\caption{Defining a tetravalent vertex formally}
\end{figure}
\begin{theorem}[\cite{Kim}]
The clasps of type $(n,0)$ satisfy the recursive relationship given in Figure $4$.
\end{theorem}
\begin{figure}
\includegraphics[scale=.45]{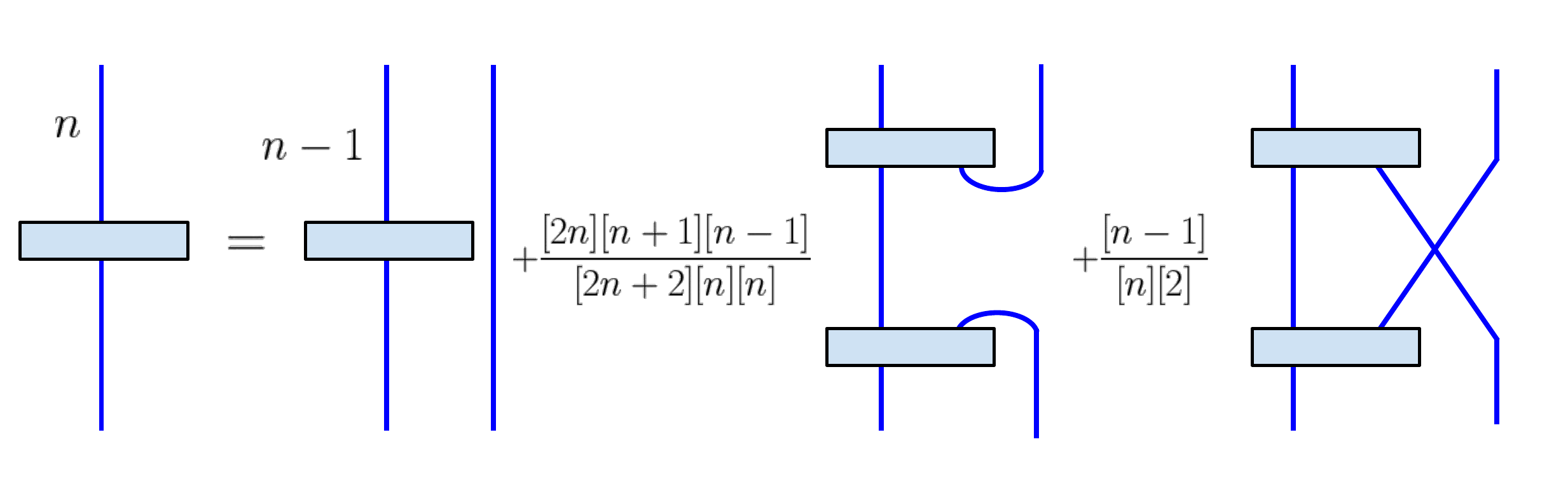}
\caption{A recursive description of a $(n,0)$ clasp}
\end{figure}
The key property of these clasps is the annihilation of webs which create a cut path with weight lower than the clasp.  This should be thought of as the generalization of the annihilation of "cups" and "caps" by Jones-Wenzl projectors. Now if a crossing is introduced  we know that it will decompose as a sum of the diagrams seen in Figure $5$.
\begin{figure}
\includegraphics[scale=.4]{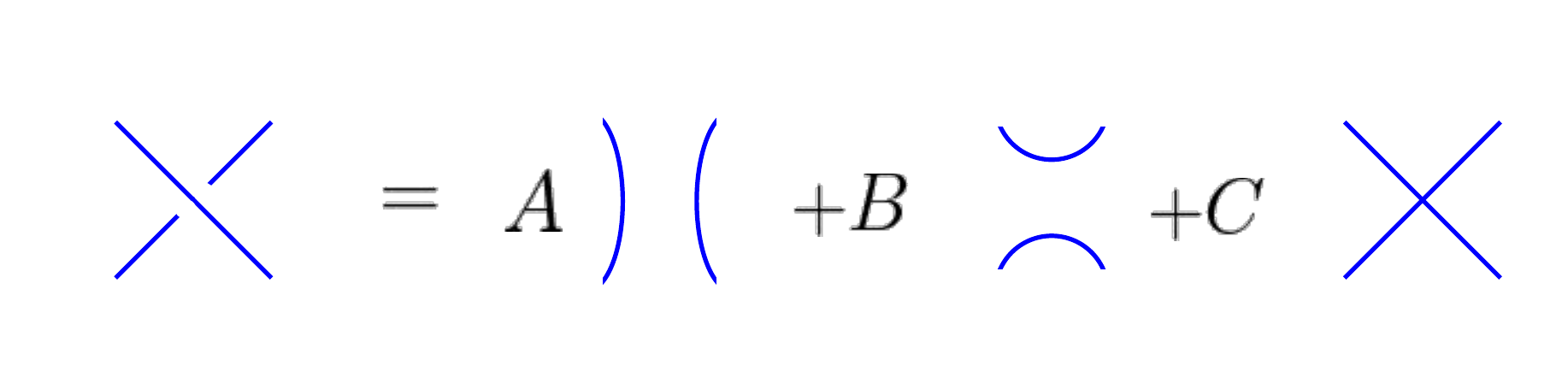}
\caption{}
\end{figure}
\begin{lem}
Let $b$ be a braid on $n$ strands and $c(b)$ be the signed crossing number of crossings of $b$.  Then
\[bP_{(n,0)}=A^{c(b)}P_{(n,0)}.\]
\end{lem}
\begin{proof}
For each crossing we can expand into the three diagrams in Figure $5$.  Then the second and third terms create a turnback which annihilates the clasp of type $(n,0)$.  This leaves us only with $A^{c(b)}$ where $A$ is the coefficient in front of the identity tangle term.  This is illustrated in Figure $6$
\begin{figure}
\includegraphics[scale=.3]{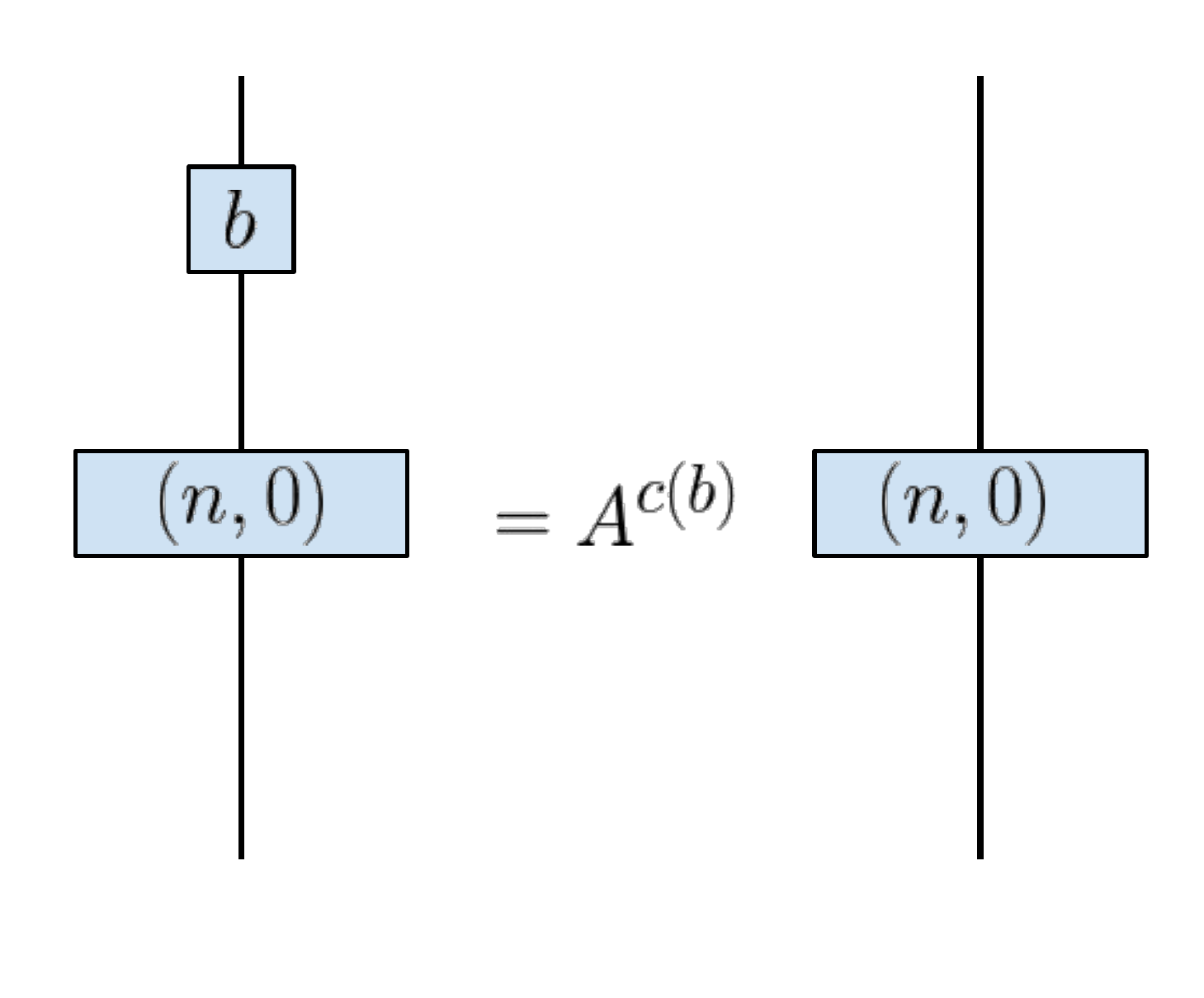}
\caption{}
\end{figure}
\end{proof}
These diagrams allow us to describe the morphisms in 
$\\Rep^{uni}(U_q(\mathfrak{sp}(4)))$ as linear combinations of diagrams where clasps are attached to boundary rather than strands. We call these diagrams clasped webs.  Of particular interest to the discussion at hand we have the following result \cite{BM}.
\begin{theorem}[\cite{BM}]
The triple clasped space, as seen in Figure $6$, 
\[I=I((a,0),(b,0),(c,0))=Hom((a,0)\otimes (b,0)\otimes (c,0),\mathbb{C})\] is either $0$ or $1$ dimensional.  We will say $\{a,b,c\}$ form an admissible triple if $a+b+c$ is even, $a+b\geq c, a+c\geq b,$ and $b+c\geq a$.  Then when $q$ is generic we have that $I$ is $1$ dimensional when $\{a,b,c\}$ are admissible and $0$ dimensional otherwise. When $q$ is a root of unity we have that $I$ is $1$ dimensional exactly when $\{a,b,c\}$ is admissible and the order of $q$ is greater than $2(a+b+c)+4$.  
\end{theorem}
\begin{proof}
This is the main result proven in \cite{BM}.  The proof follows from a recursive calculation of the corresponding theta net based on the description of clasps given by Kim.  
\end{proof}
\begin{figure}
\includegraphics[scale=.45]{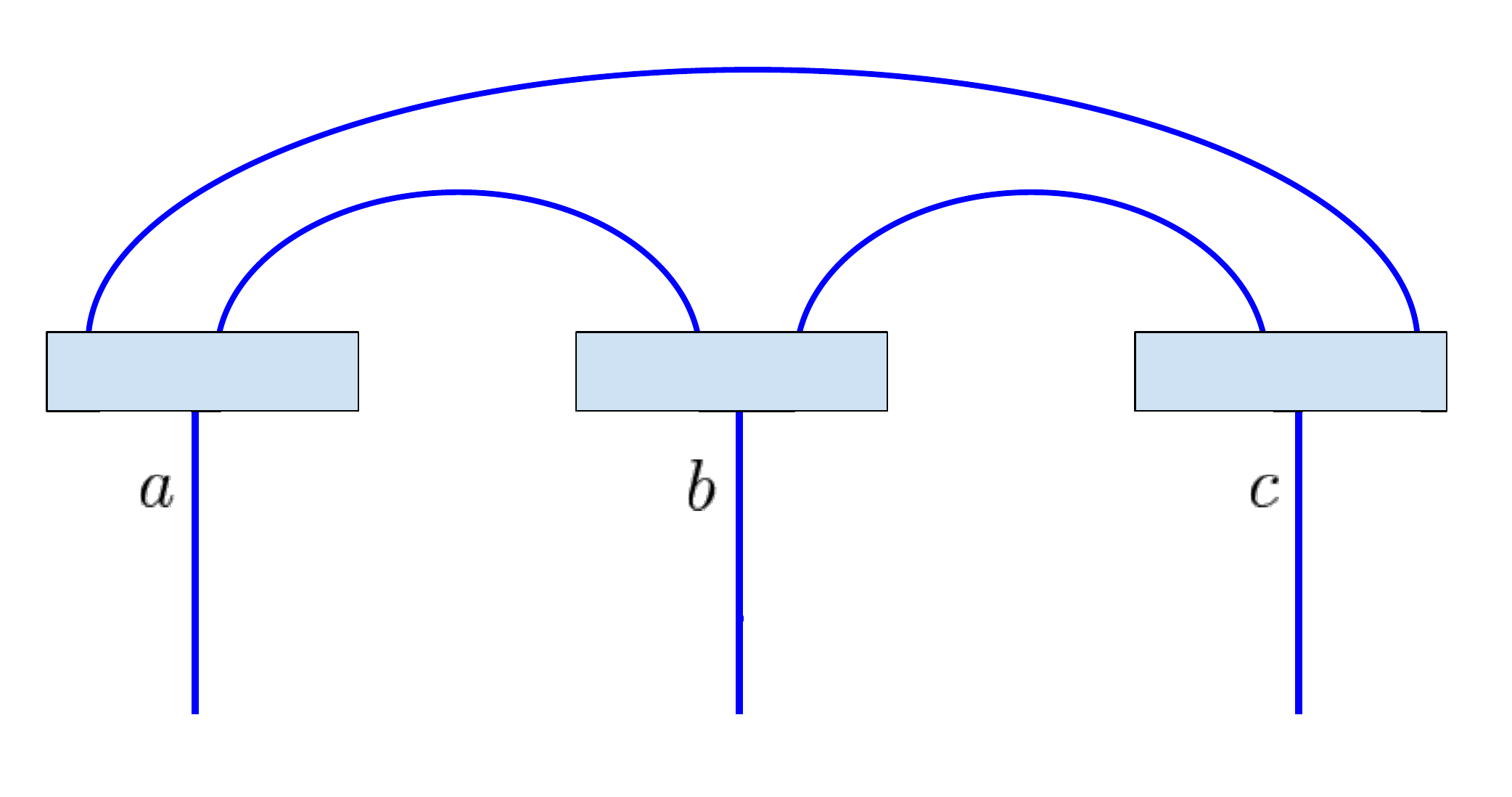}
\caption{$Hom((a,0)\otimes(b,0)\otimes(c,0),\mathbb{C})$}
\end{figure}
We also note that as this vector space is $1$ dimensional vectors in it must be a multiple of the diagram shown in Figure $6$.  
\subsection{Constructing a Modular Tensor Category} 
When $q$ is generic the spiders described above have only been described as pivotal tensor categories, with infinitely many simple objects.  Turaev has outlined the process of constructing a topological quantum field theory from a modular tensor category \cite{Tur}.  The first step in this process would be drop to finitely many simple objects.  This is done by a semi-simplification through modding out by negligible morphisms.  In the language of spiders this corresponds to modding out the clasps which have $0$ trace, where the trace is seen in Figure $8$.  
\begin{figure}
\includegraphics[scale=.35]{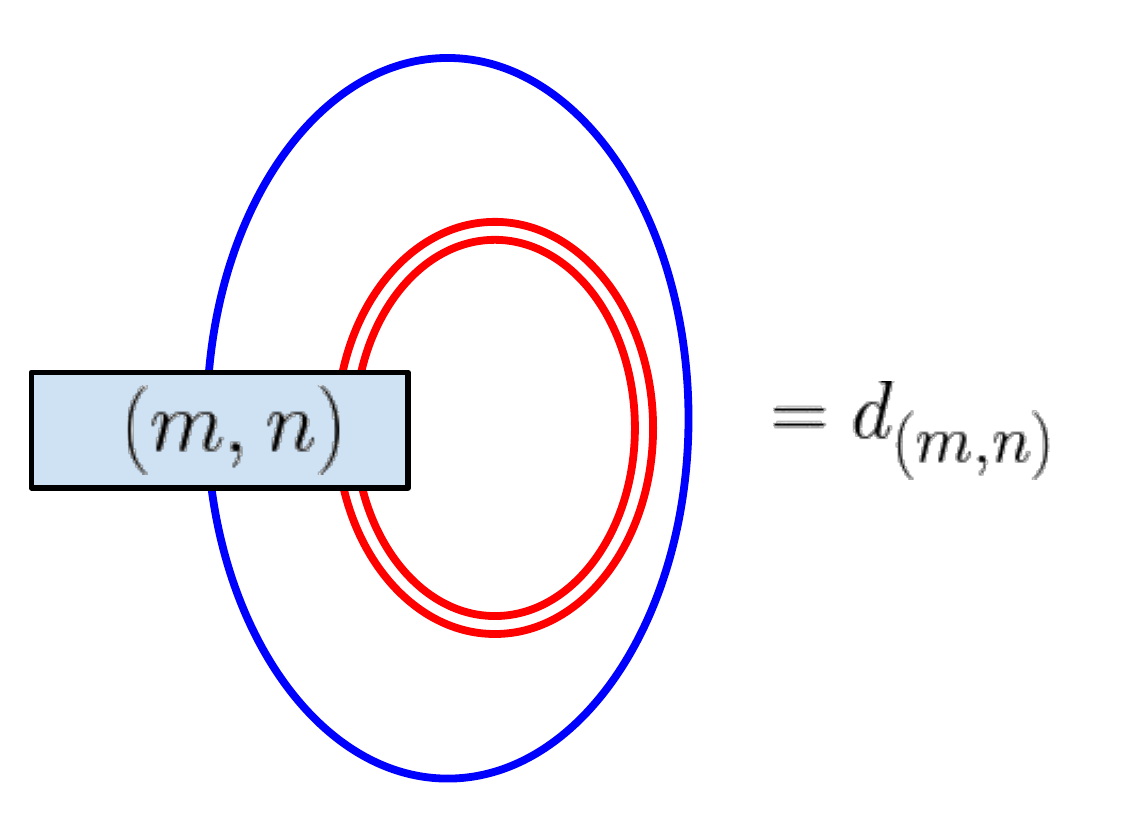}
\caption{The trace of a clasp}
\end{figure}
The details of this process were first worked out by Turaev and Wenzl \cite{TW}.  Here modular categories were constructed using the $\mathfrak{sp}(4)$ link invariant.  This work was continued by Blanchet and Beliakova, who showed the connections to the modular categories coming from quantum groups \cite{BB}.  When $q$ is a root of unity of order $2(2k+6)$ we will call the semi-simplified version of the $C_2$ spider $\mathrm{Sp}(4)_k$.  Specifically, $\mathrm{Sp}(4)_k$ has all $(a,b)$ such that $a+b\leq k$ as simple objects.  Now only will these results allow us to use the construction of Turaev, but also we have the following lemma.
\begin{lem}
In the $C_2$ spider we have that the identity tangle on $a$ strands of the first type and $b$ strands of the second type factors as the sum of clasps of weight less than or equal to $(a,b)$.  Where the ordering on clasps is the standard partial order generated by 
\[(a,b)\succ (a-2,b+1)\]
\[(a,b)\succ (a+2,b-2).\]
Moreover in $\mathrm{Sp}(4)_k$, we have that that the coefficient on the $(a,b)$ clasp is $1$ if $a+b\leq k$.  
\end{lem}
\begin{proof}
The first part of this lemma follows immediately from looking at $V^{\otimes a}\otimes W^{\otimes b}$ in $Rep^{uni}(U_q(\mathfrak{sp}(4)))$, where $V$ and $W$ are the fundamental representations.  As we are working in a modular tensor category we know that this will decompose into the direct sum of finitely many simple objects.  The second part of this lemma follows from looking at the weights of these representations.
\end{proof}
\begin{figure}
\includegraphics[scale=.45]{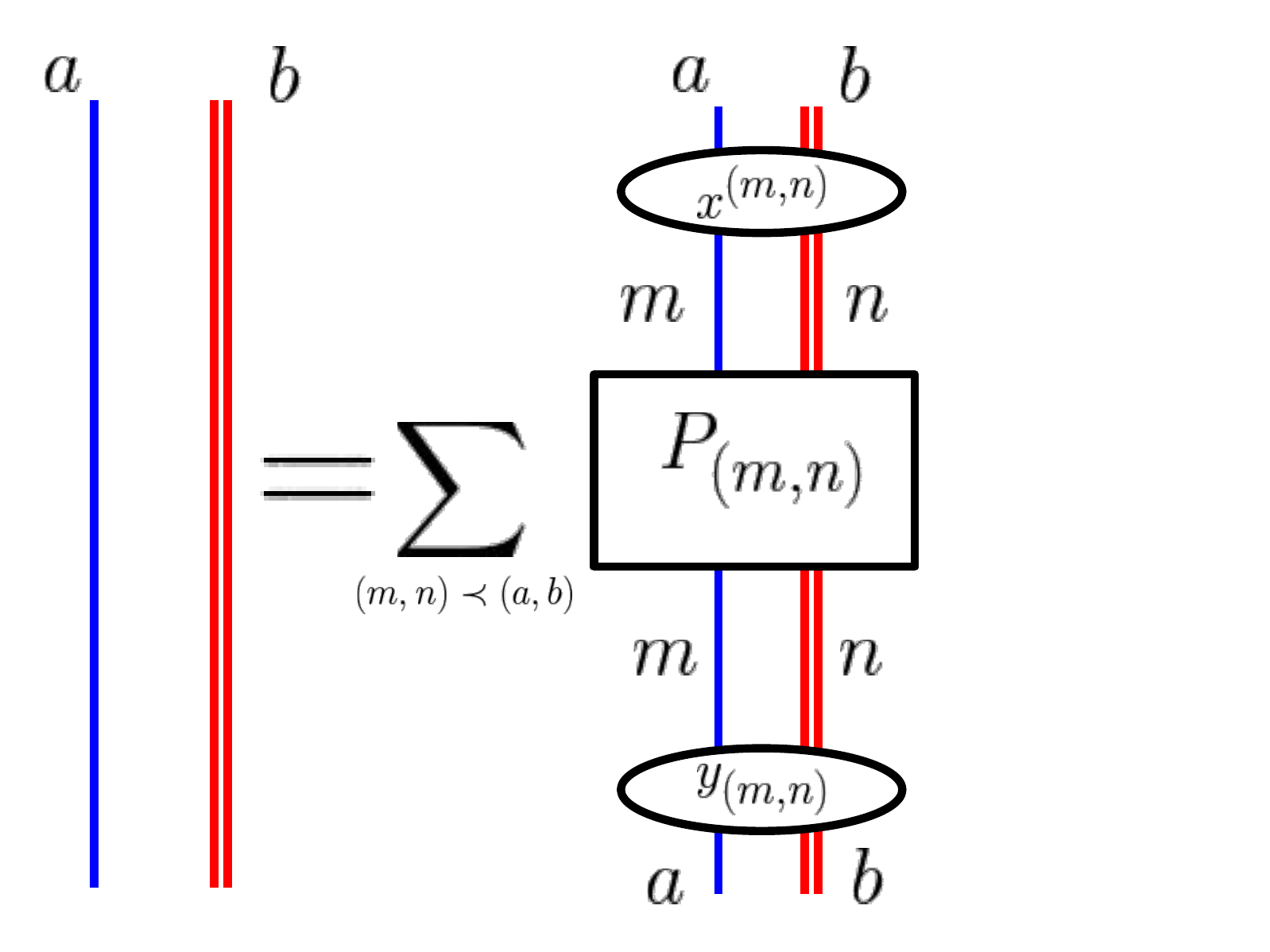}
\caption{A graphical representation of Lemma $2.1$}
\end{figure}
\section{The $\mathrm{Sp}(4)_k$ TQFT} 
\subsection{The State Space $V(\Sigma)$}
Let $\Sigma$ be a closed orientable surface.   Up to a homeomorphism into Euclidean space, we can think of $\Sigma$ as bounding a standardly embedded handlebody $H$.  Then we can associate to $\Sigma$ the spine of $H$, namely a trivalent graph whose regular neighborhood is $H$.  The choice of a particular trivalent graph corresponds to choosing a pants decomposition of $\Sigma$ by looking at the disk dual to the edge.  Now we define an admissible labeling of a trivalent graph.  At each edge of a trivalent graph we associate a clasp in $\mathrm{Sp}(4)_k$, and to each vertex a vector in the triple clasped space of the three incident edges.  Then we define $V_k(\Sigma)$, often with the $k$ omitted when it is not needed, as the free complex vector space having as a basis the admissible labelings of the above spine.
\subsection{An Action of the Mapping Class Group}
The construction of the $\mathrm{Sp}(4)_k$ TQFT not only gives rise to a vector space associated to $\Sigma$, but also an action of $MCG(\Sigma)$ on $V(\Sigma)$.  This is action arises from looking at the surgery description of the mapping cylinder associated to a mapping class.  When looking at a positive Dehn twist about a curve $\gamma$, this action is described by adjoining the curve $\gamma$ given a $-1$ framing and colored by $\omega$, where $\omega$ is a weighted sum of simple objects that allows for invariance under Kirby moves. Using the relations of the $C_2$ spider this web, with with the adjoined colored framed curve, can be resolved down to basis elements.  Then given any mapping class $h$ of $\Sigma$ we can describe
\[V_h:V(\Sigma)\rightarrow V(\Sigma)\]
by decomposing $h$ into the composition of Dehn twists.  In Figure $10$ we see an example of this action, prior to resolving, when $\Sigma$ is the genus $1$ surface.
\begin{figure}
\includegraphics[scale=.3]{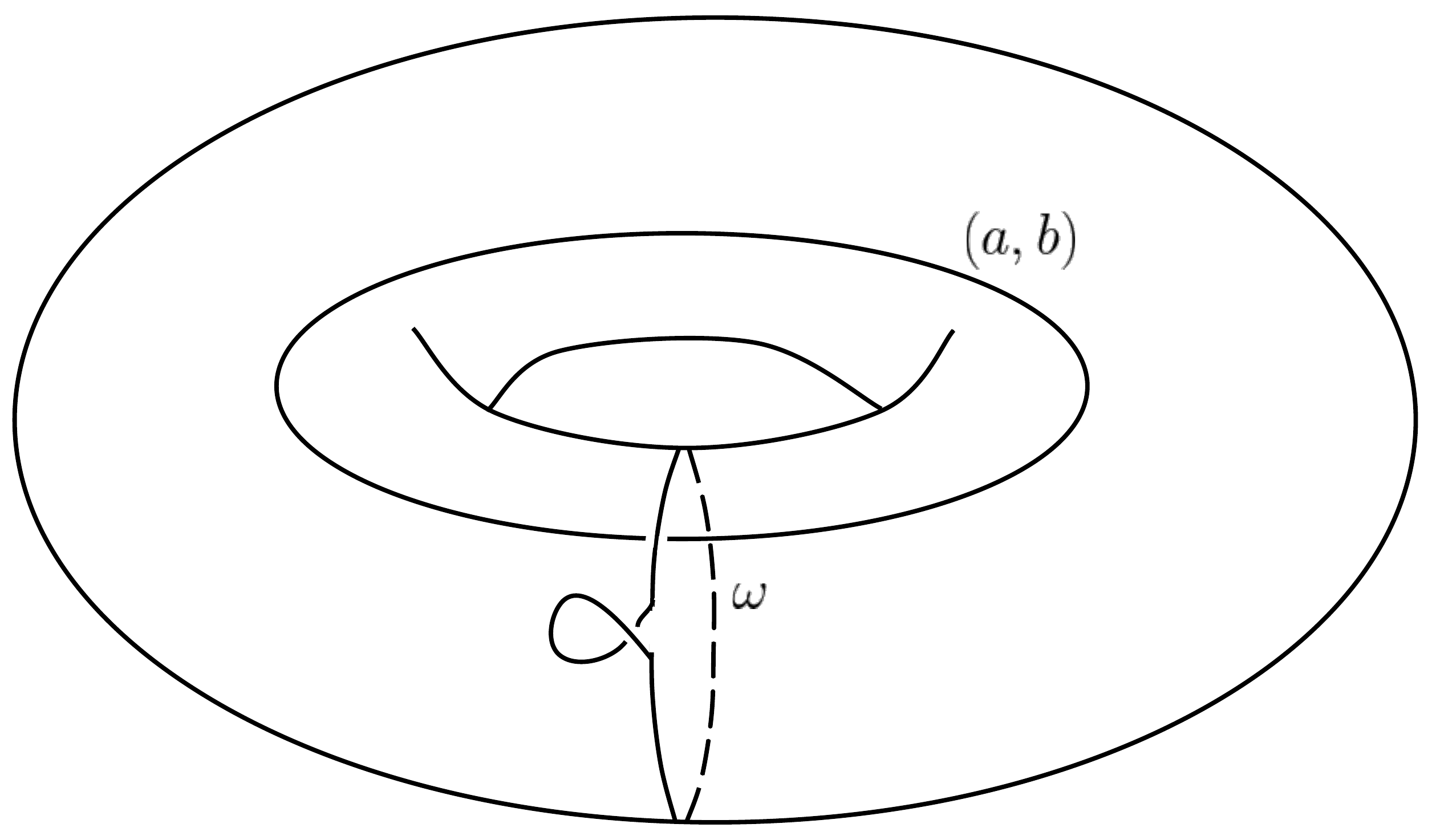}
\caption{The action of a Dehn twist about the meridian}
\end{figure}
\subsection{Curve Operators}
We look to define curve operators on the state space $V(\Sigma)$.  Let $\gamma$ be a simple closed curve on $\Sigma$.  We define a curve operator associated to $\gamma$.  Namely 
\[C(\gamma)=Z(\Sigma\times I, (\gamma)_1\times\{1/2\})\in V(\Sigma)\times V(-\Sigma)=End(V(\Sigma)).\]
Where $(\gamma)_1$ is defined to be the curve $\gamma$ colored with the fundamental representation $(1,0)$.  Then we see that $C(\gamma)$ is the resolution of adjoining the curve $\gamma$ colored with $(1,0)$.  We have the following lemma:
\begin{lem}
Let $h:\Sigma\rightarrow\Sigma$ be an orientation preserving homeomorphism.  Then we have 
\[V_h C(\gamma) V_h^{-1}=C(h(\gamma)).\]
\end{lem} 
\begin{proof}
It suffices to assume that $h$ is a positive Dehn Twist about a curve $\alpha$.  Then our above description of the action tells us we have the web made of a framed $\alpha$ colored with $\omega$ over the curve $\gamma$ over the curve $\alpha$ colored with $\omega$ framed oppositely, as seen on the left of figure $9$.
\begin{figure}
\includegraphics[scale=.35]{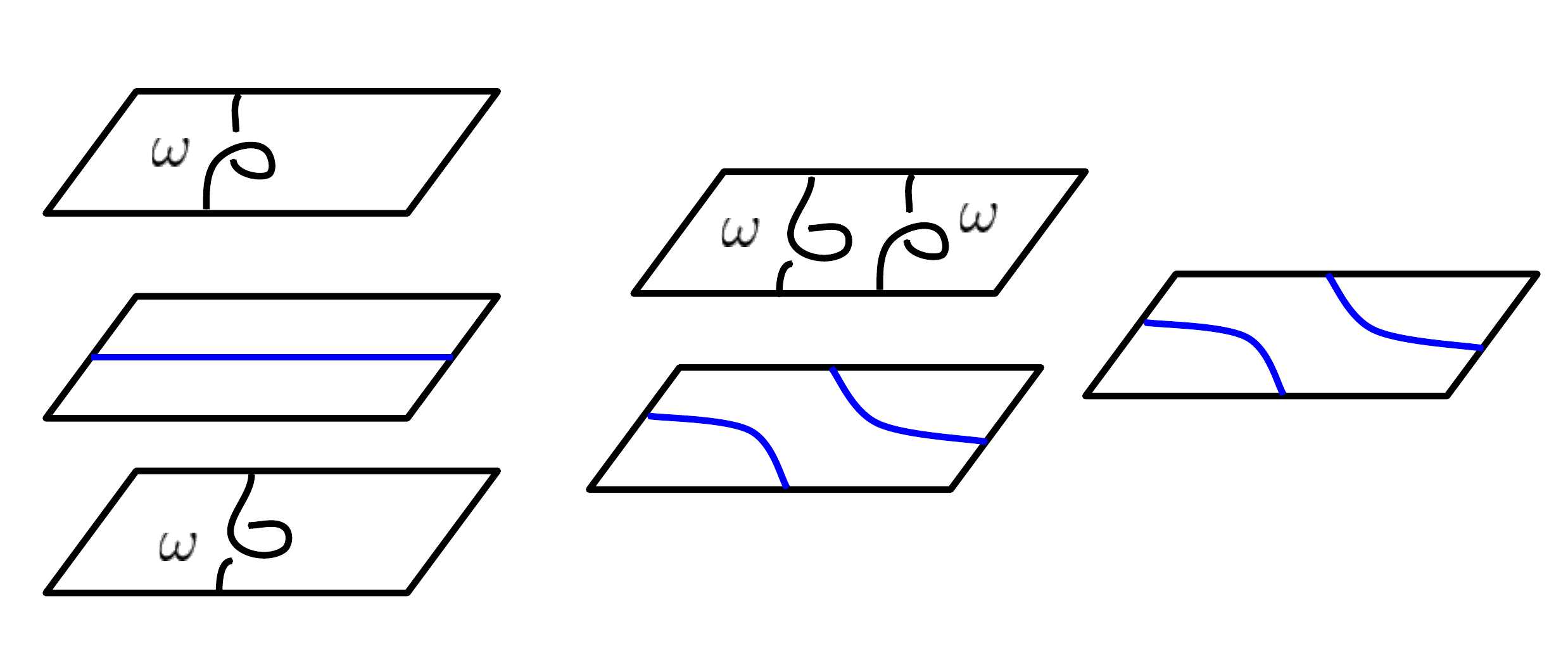}
\caption{Conjugating a Curve Operator}
\end{figure}
Then using the handle slide invariance of $\omega$ we are able to pull the bottom copy of $\alpha$ through $\gamma$ are the cost of a Dehn twist about $\alpha$.  This leaves us with $C(h(a))$ and two copies of $\alpha$ having opposite framings.  Finally we are able to apply the balanced stabilization property of $\omega$ to cancel these opposite copies.  This is visualized in Figure $11$ when $\Sigma$ is the genus $1$ surface with each parallel square having edges identified, and $\alpha$ and $\gamma$ are the meridian and longitude.  
\end{proof}
This result should be thought of as immediately following from the invariance under Kirby moves.
\section{Asymptotic Faithfulness}
\begin{lem}
Let $\alpha$ and $\beta$ be two non-trivial, non-isotopic simple closed curves on a closed orientable surface $\Sigma$.  Then there exists a pants decomposition of $\Sigma$ such that $\alpha$ is one of the decomposing curves and $\beta$ is a non-trivial "graph geodesic" with respect to the decomposition in the sense that $\beta$ does not intersect any curve of the decomposition twice in a row.
\end{lem}

This is lemma $4.1$ of \cite{FWW}.

\begin{prop}
Let $\Sigma$ be a closed, oriented surface, $h$ an orientation preserving homeomorphism, and $V_h$ the $\mathrm{Sp}(4)_k$ action.  Suppose there exists a simple closed curve $\alpha\subset \Sigma$ such that $h(\alpha)$ is not isotopic (as a set) to $\alpha$.  Then $V_h$ is a multiple of the identity for at most finitely many $k$.  That is, $h$ is eventually detected as $k$ increases.
\end{prop}

\begin{proof}
From Lemma $3.1$ we have
\[V_h C(\alpha)V_h^{-1}=C(h(\alpha)),\]
 and so it suffices to show that for some $k$, \[C(\alpha)\neq C(h(\alpha)).\]  
\par 
By the graph geodesic lemma, Lemma $4.1$ above, there exists a handlebody $H$ bounded by $\Sigma$ such that $\alpha$ bounds an embedded disk in $H$ and $h(\alpha)$ is a non-trivial graph geodesic with respect to a spine $S$ of $H$.  This is illustrated in Figure $12$, where encircled region is exactly what is forbidden by the graph geodesic property.
\begin{figure}[ht]
\centering
\includegraphics[scale=.35]{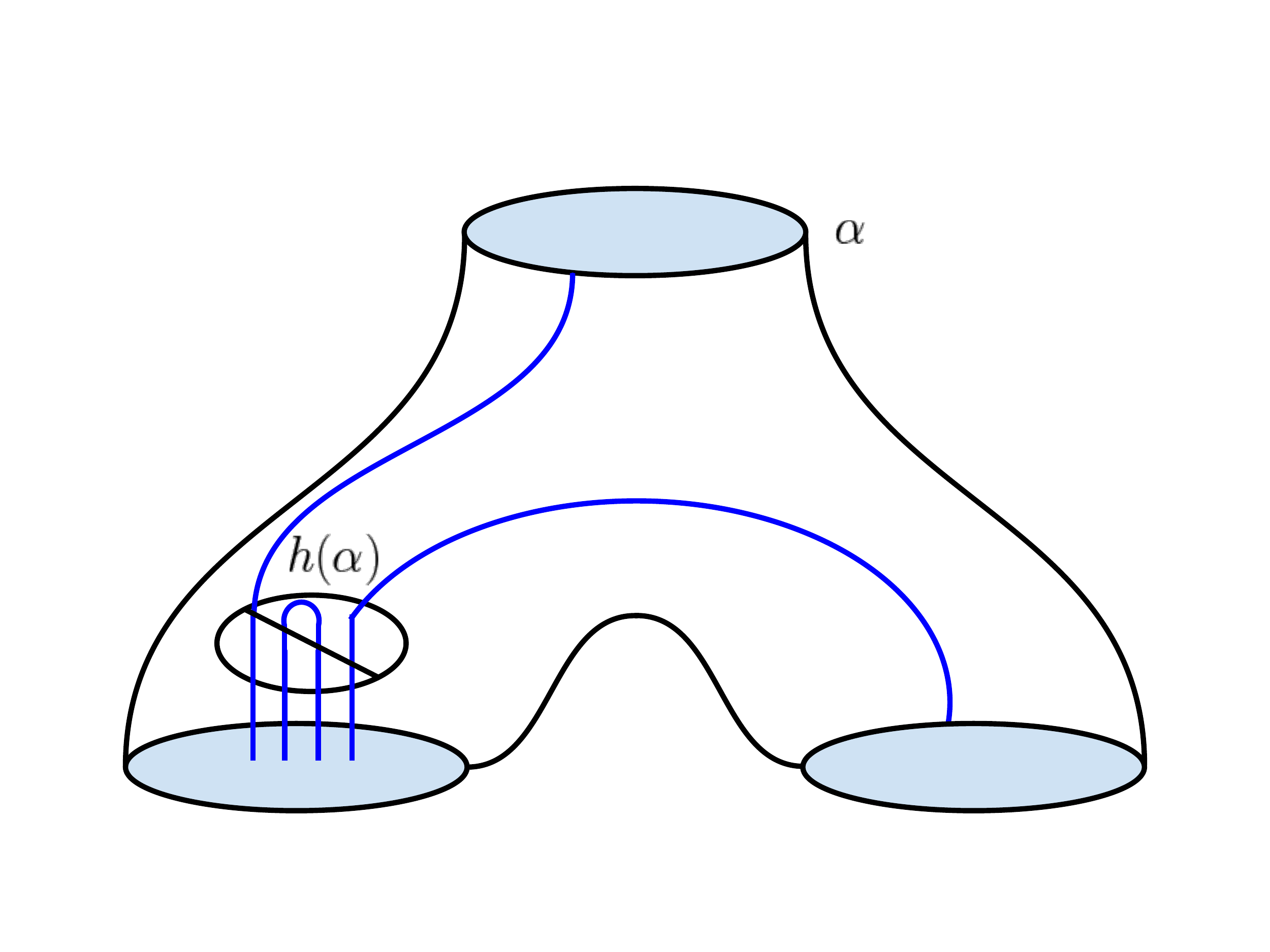}
\caption{The graph geodesic property}\label{turn}
\end{figure}
Now let $Z(H)\in V(\Sigma)$ be the vector determined by $H$ with the empty labelling, and $Z(H,h(\alpha))$ the vector determined by the pair $(H,h(\alpha))$ (where $h(\alpha)$ is pushed into the interior of $H$).  We have 
\[C(\alpha)(Z(H))=Z(H,\alpha)=dZ(H)\]
as $a$ is taken to bound an embedded disk in $H$.  It is also true that 
\[C(h(\alpha))(Z(H))=Z(H,h(\alpha)),\] 
meaning it suffices to show that $Z(H,h(\alpha))$ is not a multiple of $Z(H)$.  
\par 
We look to build a comparison vector. For each edge $e$ of the spine, let $p_e$ be the number of times that $h(\alpha)$ passes through the disk dual to $e$.  Then locally we have that the identity tangle on $p_e$ strands labelled with $(1,0)$.  
\begin{figure}[ht]
\centering
\includegraphics[scale=.4]{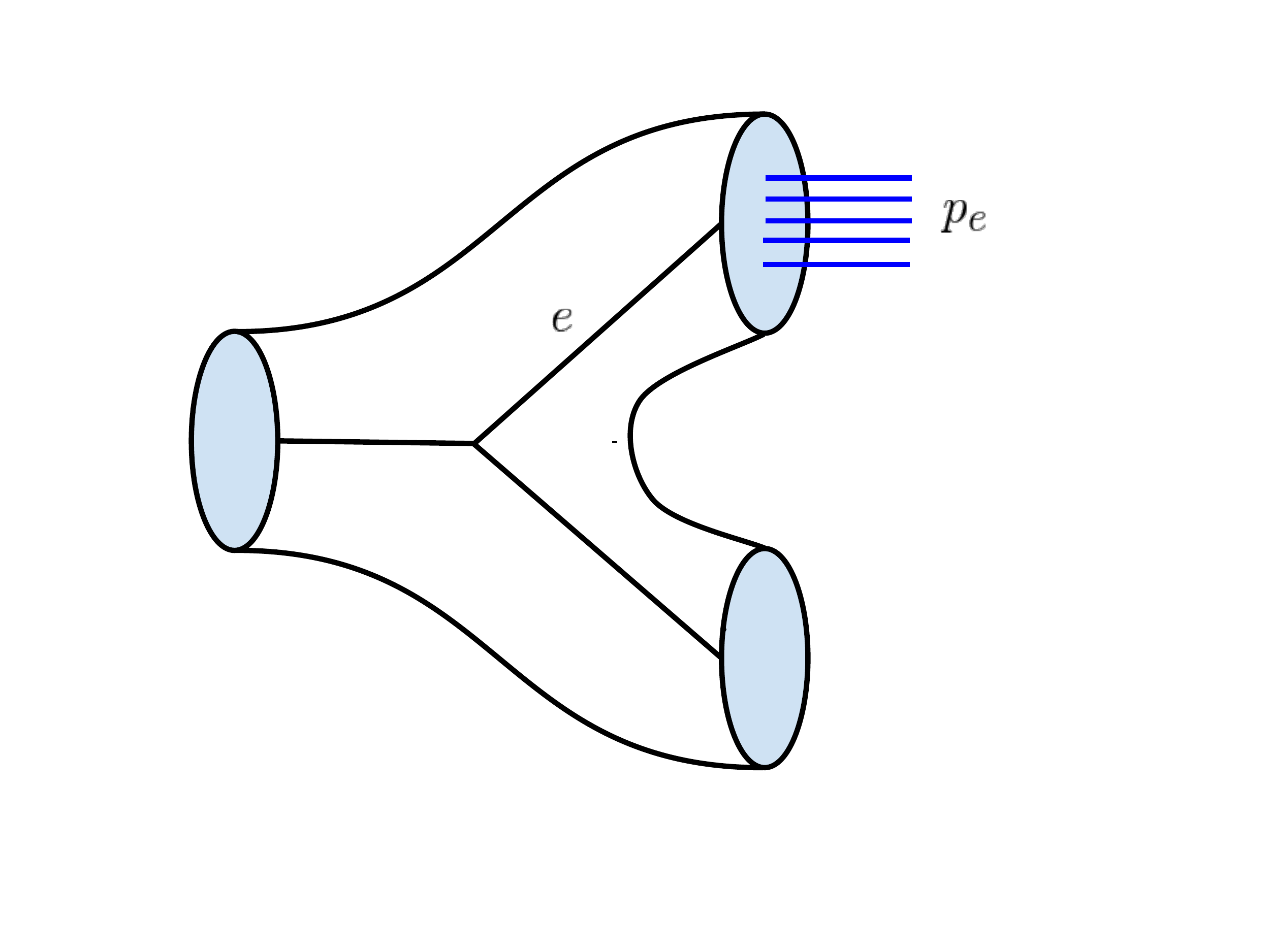}
\caption{•}\label{weight}
\end{figure}
Then our comparison labeling, $w$, of the edge $e$ is labeled by $(p_e,0)$.  Finally define the complexity of a state vector as 
\[m=\max(p_e+p_f+p_g)\}\]
where $e,f,g$ are three edges meeting at a vertex, and the max is over all vertices.  Finally, pick a level $k$ such that the order of $q$ is greater than $2m+4$.  
\par 
Let $\{b_{w}\}$ be the basis vector for $V(\Sigma)$ determined by the labellings $w$ of $S$.  We claim that 
\[Z(H,h(\alpha))= \lambda b_{w}+v,\]
where $\lambda\neq 0$, and $v$ consists of multiples of $b_x$ where $x$ is a label where 
\[(m_e,n_e)\prec (p_e,0)\]
for each each $e$ of $S$ and for $(m_e,n_e)$ the labeling of $e$ corresponding to $x$.  Applying Lemma $2.2$ we see that each edge label for $Z(H,h(\alpha))$ can be factored in the desired manner. We are left to look at the vertex of each trivalent graph.  This vertex is a skein in the triangle space 
\[I((p_e,0),(p_f,0),(p_g,0)).\]
Thus applying Theorem $2$, this skein is a multiple of the basis vector seen in Figure $7$.   We need only see that this multiple is not zero.  This particular diagram is found by pushing $h(\alpha)$ into the pair of pants about the vertex in question.  We can describe each component of $h(\alpha)$ by which boundary components it connects.  The graph geodesic lemma tells us exactly that $h(a)$ never starts and ends at the same boundary component as seen in Figure \ref{turn}. Then applying   Lemma $2.1$ we see that any braiding only contributes a nonzero scalar multiple as well.
We apply this argument to each vertex and we will have a nonzero scalar each time.  Therefore, $Z(H,h(\alpha))$ cannot be a multiple of $Z(H)$ as it decomposes as is a nonzero multiple of a basis vector and other terms and thus is not the empty labeling.  
\end{proof}

\begin{thm}
Let $\Sigma$ be a closed connected oriented surface and $\\MCG(\Sigma)$ its mapping class group. For every non-central $h\in MCG(\Sigma)$ there is an integer $k_0(h)$ such that for any $k\geq k_0(h)$ we have the operator coming from $\mathrm{Sp}(4)_k$,
\[\langle h\rangle:V(\Sigma)\rightarrow V(\Sigma),\]
is not the identity, 
\[\langle h\rangle\neq 1\in \mathcal{P}End(V(\Sigma)),\]
the projective endomorphisms.  In particular, an appropriate infinite direct sum of quantum $Sp(4)$ representations will faithfully represent these mapping class groups modulo center.
\end{thm}  
\begin{proof}
If $h$ fixes all simple closed curves then $h$ must commute with all possible Dehn twists.  As Dehn twists generate the mapping class group of any surface we have that $h$ must be in the center.  Thus for any non-central $h\in MCG(\Sigma)$ that $h(\alpha)$ is not isotopic to $\alpha$, the main theorem follows from the above proposition.  
\end{proof}

\end{document}